\documentclass[12pt]{article}
 \usepackage{a4wide}
 \usepackage{amssymb}
 \usepackage{verbatim}
 \usepackage{hyperref}
 \usepackage{xcolor,graphicx}
 \usepackage{authblk}

\usepackage[portuguese,color=blue!20, bordercolor=blue!40]{todonotes}
 \input{comandos}

\title{Boolean Percolation on Doubling Graphs }
\author[1]{Cristian F. Coletti}
\author[2]{Sebastian P. Grynberg}
\author[1]{Daniel Miranda}

\affil[1]{Universidade Federal do ABC}
\affil[2]{Universidad de Buenos Aires}

\begin{document}
\maketitle

\begin{abstract}
We consider the discrete Boolean model of percolation on graphs satisfying a doubling metric condition. We study sufficient conditions on the distribution of the radii of
balls placed at the points of a Bernoulli point process for the absence of percolation, provided that the retention parameter of the underlying point process is small enough.  We
exhibit three families of interesting graphs where the main result of this work holds.
%An interesting example of such graphs is given by the Cayley graph of the %discrete Heinsenberg group which can not be embedded in $\mathbb{R}^n$ for any %%$n$. Therefore, the absence of percolation on doubling graphs does not follow %from the subcriticality of the Boolean percolation model in $\mathbb{R}^n$ and %standard coupling arguments.
Finally, we give sufficient conditions for ergodicity of the discrete Boolean model of percolation.\\

\textbf{Keywords:} Boolean percolation, Point Processes, Doubling Spaces\\

\end{abstract}

%\kwd{Ergodicity}
%\kwd{Geometric group theory} % insert keywords separated by a semicolon

%\ams{60K35}{} % insert the primary Maths Subject Classification number in the first bracket
         % and the secondary ams number(s) in the second bracket
         % e.g. \ams{60E20}{49G03;49F10}

\section{Introduction}

The aim of this work is to study sufficient conditions for subcriticality and complete coverage in the discrete Boolean model of percolation in weighted doubling graphs. We give now an informal
description of the Boolean model of percolation. Consider a simple point process $\cX$ in
%some polish, locally compact
some suitable metric space $(\Ga,d)$. Then,
at each point of $\cX$, center a ball of random radius.  Assume that the radius are independent, identically
distributed and independent of $\cX$. Thus, $\Ga$ is partitioned in two
regions, the occupied region, which is
defined as the union of all random balls, and the vacant region, which is the
complement of the occupied region.

In this paper we consider the case in which
$\Ga$ is a doubling weighted graph equipped with the weighted graph distance, the underlying point process $\cX$ is a Bernoulli point process with
retention parameter $p$ for some $p \in (0,1)$ and the random radii are independent and identically distributed non-negative integer-valued random variables. In this setting, we prove  the
absence of unbounded connected components on the occupied region.

This model is the discrete counterpart of the Poisson Boolean model of continuum percolation
%In the  Poisson Boolean model a ball of random radius is centered at each point of a
%homogeneous Poisson point process with density $\lambda$ on $\R^n$. The corresponding radii form an independent and identically distributed collection of non-negative
%random variables which are also independent of the point process. Denote by $\cB$ the union of these balls and by $\cC$ the connected component of $\cB$ containing the origin. Let $R$ be one of the random radii and denote by $\textbf{P}$ the law governing
%the continuous boolean model. Also, denote by $\bf{E}$ the corresponding expectation operator. In \cite{hall1985continuum}, Hall proved that
%for values of $\lambda$ small enough, $\cC$ is almost surely bounded provided that $\E[R^{2n-1}]$ is finite. In \cite{roy}, Meester and Roy proved that if $n \geq 2$, then the %expected number of balls in the occupied component  which contains the origin is finite whenever $\lambda$ is small enough if, and only if,
%$\E[R^{2n}]$ is finite. Also, they proved that if $\E[R^{2n-1}]$ is finite then $\P(\mbox{number of balls in any occupied component is finite})=1$ provided that $\lambda$ is small
%enough.  In \cite{Gouere}, Gouere showed that the set $\mathcal{C}$ is almost surely bounded for small enough $\lambda$ if and only if $\E[R^n]$ is finite.
%
%The Boolean model of percolation on $\mathbb{R}^n$
which belongs to the family of con\-ti\-nuum percolation models. In fact, the history of continuum percolation began in 1961 when W. Gilbert
\cite{gilbert1961random}
introduced the random connection
model on the plane. In $1985$, S. Zuev and A. Sidorenko
\cite{zuev1985continuous} considered continuum models of percolation where
points are chosen randomly in space and surrounded by shapes which can be random
or
fixed. In that work the authors studied the relation between critical parameters
associated to that model. For a comprehensive study of continuum models of percolation,
see the book of R. Meester and R. Roy \cite{roy}. To the best of our knowledge, this is the first time that the discrete Boolean model of percolation appears in the literature.

\begin{comment}
In $2001$, I. Benjamini and O. Schram considered different percolation mo\-dels
in the hyperbolic plane and
on regular tilings in the hyperbolic plane. Recently, the Boolean model of
percolation has received considerable a\-tten\-tion. In \cite{tykesson}, J.
Tykesson studied the Poisson
Boolean continuum model for percolation in $\mathbb{H}^n$. He showed that, for certain values of the intensity of the underlying Poisson
point process, there are infinitely many unbounded components in the occupied and vacant regions.  In $2014$, C. Coletti and
S. Grynberg  \cite{cristiangrynberg} used the existence of a subcritical phase for the discrete Boolean model of percolation on the $n$-dimensional
integer lattice $\mathbb{Z}^n$ to construct, forward in time, interacting particle systems with generator admitting a Kalikow-type decomposition. To the best of our knowledge, this is the first time that the discrete Boolean model of percolation appears in the literature.

We finish this introduction by pointing out
\end{comment}

We point out that that the discrete Boolean model of percolation in graphs where the underlying point process is a Bernoulli point process and the balls
under consideration are (closed) balls of radius $1/2$ (any positive number lower than one  would work as well) co\-rres\-ponds to the
case of independent site percolation.

\begin{comment}
In fact, the theory of discrete percolation began before the theory of continuum percolation
with the paper of S. Broadbent and J. Hammersley in $1957$ where they
introduced iid bond percolation. Their motivation was to understand the flow of
a liquid through a porous media. For a long time, percolation on the
$n$-dimensional integer lattice concentrated the attention of the probabilistic
and physical community. In $1989$,  percolation theory underwent a remarkable
change with the work of R. Lyons. He studied percolation on regular tree or
tree-like graphs. In the nineties appeared the first results on percolation on
graphs beyond $\mathbb{Z}^n$ and tree-like graphs. In $1996$, I. Benjamini and
O. Schramm proposed a comprehensive study of percolation on Cayley graphs.
O. H\"aggstr\"om \cite{haggstrom2006uniqueness} and A. Procacci
\cite{procacci2004infinite} proposed to study percolation problems on general
graphs.
\end{comment}
In this paper we prove that if the underlying graph satisfies a doubling condition and if the family of random radii are i.i.d. random variables with finite $\asdim$-moment, where $\asdim$ is the corresponding Assouad dimension which will be defined in section \ref{dnsr}, then the connected components  arising in the discrete Boolean model are almost surely finite for sufficiently small values of $p$. We also prove that such behavior does not occur if the random radii have infinite $\asdim$-moment.

We remark that the absence of percolation on doubling graphs does not follow from the subcriticality of the Boolean percolation model in $\mathbb{R}^n$ and standard coupling arguments. Indeed, an interesting example where the main result of this work holds is given by the Cayley graph of the discrete Heisenberg group which can not be embedded in $\mathbb{R}^n$ for any $n$.

%Here is the outline of the paper. In section \ref{dnsr} we describe the discrete Boolean model of percolation and state the main result of this work (Theorem \ref{T1}),
%namely the absence of percolation provided that the retention parameter of the underlying point process is small enough. In subsection \ref{doubling_graphs} we provide some %interesting examples of graphs satisfying the assumptions of Theorem \ref{T1}. Indeed, Theorem \ref{T1} is valid for the family of graph with polynomial growth. In particular, it  %holds for lattices and nilpotent groups. In section \ref{mainp} we prove Theorem \ref{T1}. In subsection \ref{infr} we prove that the whole space is covered if the size of a random %ball has infinite expectation. Finally, in section \ref{ergodicitynumberofinfiniteclusters}, we address the problem of determining sufficient conditions for ergodicity of the Boolean %discrete percolation model.

%The examples include the family of graph with polynomial growth. This includes, for instance, lattices and cayley graphs of nilpotent groups.

%The article is organized as follows. In Section 2, we introduce the notation, the definitions and state the main theorems. In Section 3 we present some elementary examples which %justify the definitions proposed.  In Sections 4, 5, we prove the main results. Finally, in Section 6, we prove some results on the trace of Markov processes needed in the article and %which we did not find in the literature.

This paper is organized as follows. In section \ref{dnsr} we describe the discrete Boolean model of percolation and state the main theorems. In subsection \ref{doubling_graphs} we provide some interesting examples of graphs satisfying the assumptions of the main theorems. These examples include graph of polynomial growth such as self-similar graphs and Cayley graphs of nilpotent groups. In section \ref{mainp} and section \ref{infr} we prove the main results. In section \ref{ergodicitynumberofinfiniteclusters} we address the problem of determining sufficient conditions for ergodicity of the discrete Boolean model of percolation.

\section{Definitions, notation and statement of the main results}\label{dnsr}
\subsection{Doubling Graphs}

 Throughout this paper $\mathbb{N}$ will denote the set of natural numbers,  $\mathbb{N}_0$ will denote the set of non-negative integer numbers and $\Ga= \left(\mathcal{G}, \mu\right)$ will denote a  weighted graph, where $\mathcal{G}=(V,E)$ is a countably infinite connected graph. Here $\mu_{xy}$ is a non-negative function on $V \times V$ such that: (i) $\mu_{xy}=\mu_{yx}$; and (ii) $\mu_{xy} > 0$ if and only if $x \sim y$ (i.e. if and only if $x$ and $y$ are neighbors in $\Ga$). .
 
 A \emph{path} on $\cG(\cX,\cR)$ is a sequence of distinct vertex $v_0, v_1,\dots, v_n$ with $v_{i-1} \neq v_i$ such that $\{v_{i-1},v_i\}\in\cE$, $i=1,\dots, n$. The length of a path from $x$ to $y$ $(z_{0}=x,z_{1},\ldots,z_{n}=y)$
 is defined as $\sum_{i=1}^n \mu_{z_{i-1}z_i}$. Finally the
distance between two points $x,y$ denoted by   $d(x,y)$ is the length of a shortest path from $x$ to $y$.  We regard $\Ga$ as a metric space with the metric $d$ given by the weighted distance on the graph $\Ga$.
% A path from $x$ to $y$ is \emph{geodesic} if its length is $d(x,y)$.

% Thus, for any finite set $A\subset
% V$,
% \begin{eqnarray}
% \medida{A} = \displaystyle \sum_{y \in A} \mu(y), \nonumber
% \end{eqnarray}
% where $\mu(x)=\displaystyle \sum_{y: y \sim x} \mu_{xy}$ with $\mu\left(\displaystyle \emptyset \right)=0$.

Also, $B(v,r)=\{u\in V:\,
d(u,v)\leq r\}$ denotes the closed ball of radius $r$ centered at $v$ and
$S(v,r)=\{u\in V:\, d(u,v)=r\}$ denotes the sphere of radius $r$ in $\Ga$ around
$v$.

We write $\medida{\cdot}$ for the counting  measure on $V$. We note that $\medida{B(v,r)}$ depends on the distance on $\Ga$ and consequently on the weights on the graph.

%Throughout this paper, $\mathbb{N}_0$ will denote the set of non-negative integer numbers and $\Ga=(V,E)$ will denote a countably infinite connected graph. We regard $\Ga$ as a metric space with the metric $d$ given by the graph distance on $\Ga$.
%We write $\medida{\cdot}$ for the counting measure on $V$. Thus, for any finite set $A\subset
%V$, $\medida{A}$ denotes the number of elements of $A$. Also, $B(v,r)=\{u\in V:\,
%d(u,v)\leq r\}$ denotes the closed ball of radius $r$ centered at $v$ and
%$S(v,r)=\{u\in V:\, d(u,v)=r\}$ denotes the sphere of radius $r$ in $\Ga$ around
%$v$.

\begin{defi}
 We say that $(\Ga,d)$ is a \emph{doubling metric graph} if there exists a non-negative constant $C$
such that any ball $B$ in $\Ga$ can be covered with at most $C$ balls whose
radius is half the radius of $B$.
\end{defi}

 For further reading
on doubling metric spaces, see \cite{gromovmetric} and \cite{mackay2010conformal}. Doubling graphs arise naturally in many applications. For
instance, metric embedding of  doubling graphs turned out to be useful for
algorithm design. 
%The doubling condition has also appeared in the study of the
%problem of designing routing algorithms for networks with structure parametrized
%by its doubling dimension, see e.g. \cite{chan2005hierarchical} and \cite{talwar2004bypassing}.
A related but stronger  condition is the volume and time doubling assumption on
graphs which has been used to
prove upper and lower off-diagonal, sub-Gaussian
transition probability estimates for strongly recurrent random walks, see \cite{telcs2001volume}. In section \ref{exampledoubling} we give  many examples of doubling graphs.

In \cite{Doyle}, Peter G. Doyle applies Rayleigh's short-cut method to prove P\'olya's recurrence theorem. In that work, the author uses the notion of doubling graphs in the proof of P\'olya's theorem for the $3$-dimensional lattice. It is worth mentioning that in \cite{Benjamini2011} the authors show how to construct planar graphs with the doubling property. Indeed, in that work the authors construct, for any $\alpha > 1$, a triangulation of the plane for which every ball of radius $r$ has (up to a multiplicative constant) $r^{\alpha}$ vertices.

\bigskip

\noindent {\bf{Assouad Dimension}} Now we introduce the concept of Assouad dimension which will be used to state our main result. To begin with, let $\epsilon > 0$ be given. We call a subset $A\subset V$
\emph{$\epsilon$-separated} if $d(v,w)\geq\epsilon$ for all distinct $v,w\in A$. Let $N(B,\epsilon)$ be the maximal cardinality of an $\epsilon$-separated subset of $B$. Then
$(\Ga,d)$ is doubling if and only if there exists $C<\infty$ such that

\begin{equation}
\label{1.4.5}
N(B(v,r), r/2)\leq C
\end{equation}
for all balls $B(v,r)$ in $\Ga$. An easy inductive argument lets us show that if \reff{1.4.5} holds, then there exists $C'$ and $\beta$ depending only on $C$ such that

\begin{equation}
\label{1.4.6}
N(B(v,r), \epsilon r)\leq C'\epsilon^{-\beta}
\end{equation}
for any ball $B(v,r)$ in $\Ga$ and any $0<\epsilon<1$. For instance, we may take $\beta=\log_2C$. Thus, the doubling condition is a finite-dimensional hypothesis which controls the
growth of the cardinalities of separated subsets of any ball at any scale and location. We now give a formal definition of Assouad dimension

\bdefi
Let $\operatorname{Cover}_{\Gamma}$ denote the set of all $\beta > 0$ for which there exist $C_{\beta} > 0$ such that, for any $B(v,r)$ in $\Gamma$ and for any $0 < \epsilon < 1$ we have $N(B(v,r), \epsilon r)\leq C\epsilon^{-\beta}$. The Assouad dimension of $\Gamma$ is then defined to be
\[
\asdim = \inf \{\beta \in \operatorname{Cover}_{\Gamma}\}.
\]

\edefi

%Finally, we define the {\it{Assouad dimension}} as the infimum of all $\beta > 0$ such that
%(\ref{1.4.6}) holds. Denote the Assouad dimension of $\Ga$ by $\asdim$.

For further details see \cite{assouad1}, \cite{assouad2} and references therein. A useful fact about the Assouad dimension which follows directly from its definition and which will be used later is that it is possible to control the volume of any ball in terms of $\asdim$.
It follows from (\ref{1.4.6}), by taking $\epsilon=1/r$, that there exists a constant $C_1$ depending only on $\asdim$ such that
\begin{equation}
\label{constant}
\medida{B\left(v,r\right)} \leq C_1 r^{\asdim}
\end{equation}
for any $v \in V$ and any $r \in \mathbb{N}$.

\subsection{Marked Point Process}
A \emph{Bernoulli point process} on $\Ga$ with retention parameter $p\in(0,1)$, is a family of independent $\{0,1\}$-valued random variables $\cX=(X_v:\, v\in V)$ with common law $\P$ such that $\P(X_v=1)=p$. Identify the family of random variables $\cX$ with the random subset $\cP$ of $V$ defined by $\cP=\{v\in V:\, X_v=1\}$ whose distribution is a product measure whose marginals at each vertex $v$ are Bernoulli distribution of parameter $p$.

By a \emph{Bernoulli marked point process} on $\Ga$ we mean a pair $(\cX,\cR)$ formed by a Bernoulli point process $\cX$ on $\Ga$ and a family of independent, identically distributed $\N_0$-valued random variables $\cR=(R_v: v\in V)$ called marks. We assume that these marks are independent of the point process $\cX$.

Let $(\cX,\cR)$ be a marked point process on $\Ga$ with retention parameter $p$ and marks with common distribution $\P_{\rho}$. We denote by $\E_{\rho}$ the expectation operator induced by $\P_{\rho}$. Also, we denote by $\P_{p,\,\rho}$ and $\E_{p,\,\rho}$ respectively the probability measure and the expectation operator induced by $(\cX,\cR)$.

\subsection{Random Graphs and Percolation} Let $(\cX,\cR)$ be a Bernoulli marked point process on $\Ga$. Then we define an associated random graph $\cG(\cX,\cR)=(V,\cE)$ as the undirected random graph with vertex set $V$ and edge set $\cE$ defined by the condition $\{v,w\}\in\cE$ if, and only if, $X_v=1$ and $w\in B(v,R_v)$ or $X_w=1$ and $v\in B(w,R_w)$.

A set of vertices $C\subset V$ is connected if, for any pair of distinct vertices $v$ and $w$ in $C$, there exists a path on $\cG(\cX,\cR)$ using vertices only from $C$, starting at $v$ and ending at $w$. The connected components of the graph $\cG(\cX,\cR)$ are its maximal connected subgraphs.

The cluster $C(v)$ of vertex $v$ is the connected component of the graph $\cG(\cX,\cR)$ containing $v$. Define the {\it{Percolation event}} as follows:
\begin{eqnarray}
\left(\mbox{Percolation}\right):=\bigcup_{v\in V}\left(\medida{C(v)}=\infty\right).
\end{eqnarray}

\subsection{Main Results}
Now we state the first main result of this work.

\bteo
\label{T1}
Let $\Ga$ be a doubling graph. Let $(\cX,\cR)$ be a marked point process on $\Ga$ with retention parameter $p$ and marks with common distribution $\P_{\rho}$. Let $R$ be a random variable whose distribution is $\P_{\rho}$. If
\begin{eqnarray}
\label{E100f}
\E_{\rho}\left[R^{\asdim}\right]<\infty,
\end{eqnarray}
then there exists $p_0>0$ such that $\P_{p,\,\rho}($Percolation$)=0$ for all  $p\leq p_0$.
\eteo

\paragraph{Complete Coverage} We complement the result of Theorem \ref{T1} by establishing a sufficient condition for complete coverage of $V$. For any $A\subset V$, define $\Lambda(A)=\bigcup_{v\in A\cap\cP}B(v,R_v)$.

We say  that the growth of a given graph is \textit{at least polynomial} if there exist constants $C$ and $d$ such that $\medida{B(v,r)}\geq C r^d$ for any $v\in V$ and $r>0$. Here $B(v,r)$ denotes the closed ball centered at $v$ of radius $r$ in $\Ga$.  Analogously we say  that the growth of a given graph is \textit{at most polynomial} if there exist constants $C^{\prime}$ and $d^{\prime}$ such that $\medida{B(v,r)}\leq C^{\prime} r^{d^{\prime}}$ for any $v\in V$ and $r>0$.

For an  almost-transitive  graph $\Ga$, i.e,  if the automorphism group of $\Ga$ acts on it with finitely many orbits we have the following equivalence \cite{Woessalmosttransitive}: if $\Ga$ is an almost-transitive graphs whose growth is at most polynomial then  its growth is at least polynomial. Indeed, almost-transitive graphs whose growth is at most polynomial are doubling graphs.

% Precisely,  there exist constants $C_1, C_2 > 0$ such that for any $r$ and $v \in V$,
% \begin{equation}\label{almost-quasi-transitive}
% C_1 r^d \leq \medida{B(v,r)} \leq C_2 r^d.
% \end{equation} 
% \todo{desigualdade apenas por baixo?}

\bteo
\label{Riidinf}
Let $\Ga$ be an infinite, locally finite, 
%almost-transitive 
graph whose growth is at least polynomial: $\medida{B(v,R)} \geq C_1 R^{C}$. Let $(\cX,\cR)$ be a marked point process on $\Ga$ with retention parameter $p$ and marks distributed according to the probability distribution $\P_{\rho}$. Let $R$ be a random variable whose law is $\P_{\rho}$. If $\E_{\rho}\left[R^{C}\right] = \infty$, then for any $p\in(0,1]$, $\Lambda(V)=V \ \P_{p,\,\rho}$-almost surely.
\eteo

It follows from the discussion above  that Theorem \ref{T1} and Theorem  \ref{Riidinf} hold for the family of almost-transitive graphs under a suitable moment condition.

In the discrete space version of the Boolean model of percolation considered in this work we pursue a geometric approach inspired by geometric arguments considered in \cite{Gouere}. There are considerable technical differences in the present work, though. The main ones are as follows: a) the growth of random balls is controlled by assuming a moment condition involving the Assouad dimension of the underlying graph and the random radii of the corresponding balls; this requires considerably more care in the proof of the main result of this work, and b) the ergodicity of the discrete Boolean model of percolation follows under the requirement that the underlying graph admits a family of symmetries which acts separating points.

\paragraph{Phase transition} Consider the discrete Boolean model of  percolation  introduced above.  As in other percolation models we define the percolation probability $\theta\left(p\right)$ by $\theta\left(p\right) = \textbf{P}_{p,\rho}\left(\mbox{Percolation}\right)$. A standard coupling arguments gives the monotonicity of $\theta\left(p\right)$ in $p$. Thus, the critical parameter
\begin{equation}
\pc\left(\Ga\right) = \sup \{p : \theta\left(p\right) = 0\}
\end{equation}
is well defined. 

Now replace the random radii in this model by the deterministic radius $1/2$. What we get is the independent site percolation model in $\Ga$.Then, a direct coupling with the site percolation model in $\Ga$ yields $\pc\left(\Ga\right) \leq p_c^{s}\left(\Ga\right)$, where $p_c^{s}\left(\Ga\right)$ is the critical parameter for independent site percolation in $\Ga$.

\begin{teo}\label{LyonsandYuval}
If $\Ga$ is the Cayley graph of a group $G$ with at most polynomial growth containing a subgroup isomorphic to $\mathbb{Z}^2$, then $\pc\left(\Ga\right) < 1$.
\end{teo}

The proof of the previous theorem follows from the relation stated above between the critical parameters for independent site and Boolean percolation and Theorem $7.17$ and Corollary $7.18$ in \cite{yuval} where the authors proved that $p_c^{s}\left(\Ga\right) < 1$ for graphs satisfying the assumptions of Theorem \ref{LyonsandYuval}. Therefore, we may use Theorem \ref{T1} and Theorem \ref{LyonsandYuval} to prove that phase transition occurs for the discrete Boolean model of percolation.
%which state the analogous result to that stated in Theorem \ref{LyonsandYuval} above for independent site percolation in graphs.
See remark \ref{RemarkPansu} in subsection \ref{exampledoubling} for an application of Theorem \ref{LyonsandYuval}.

\subsection{Examples of Doubling Graphs}\label{exampledoubling}
\label{doubling_graphs}
%Let us give some examples of doubling graphs
The first and fundamental example of a doubling metric graph is  $\bbZ^n$
which has polynomial growth of order $n$. This example may be generalized, at least, in three ways: graphs with polynomial growth, Cayley graphs of nilpotent groups and self-similar graphs.

We observe that an easy way to obtain other examples of doubling graphs is to consider subgraphs of a doubling graph, since the property of a graph  being  doubling is hereditary, i.e., a subspace of a doubling metric space  is doubling (\cite{gromovmetric}, B.2.5).

\paragraph{Graphs with Polynomial Growth}
A large family of doubling metric graphs is the one composed of graphs with polynomial
growth. For a comprehensive study of such graphs see \cite{imrich1991survey} and
\cite{gromov}. For each $v\in V$, the
\emph{growth function} $\ga(v,\cdot):\N\to\N_0$ with respect to the vertex $v$, is given by
\begin{eqnarray}
\label{GFV}
\ga(v,r):=\medida{B(v,r)}.
\end{eqnarray}
It is worth mentioning that, in the case of transitive graphs the growth
function does not  depend on the choice of a particular vertex $v$ and in the
case of non-transitive graphs, the growth
functions for two different vertexes differs only by a multiplicative constant.
%  Thus, we may denote the %order
% growth function $\ga(v,r)$ just by $\ga(r)$.

We assume that for each $r\in\N$ $\ga(r)=\sup_{v\in V}\ga(v,r)<\infty$. We say  that a given graph has \textit{polynomial growth} if there exist
constants $C$
and $d$ such that its associated growth function satisfies the inequality $C^{-1} r^d \leq \ga(v,r)\leq C r^d$ for any $v\in V$ and $r>0$.

A graph with polynomial growth satisfies a condition which is slightly
stronger than being a doubling metric space. It may be proved that in this case the graph is a
doubling measure space. To keep the paper self-contained we give the definition of doubling measure space.

\begin{defi}
Let $\left(M,d\right)$ be a metric space. A positive Borel measure space $\mu$ on $M$ is said to be doubling if there exists a constant $C > 0$ such that
\begin{equation}
\mu(2B) \leq C \mu (B) \nonumber
\end{equation}
for all balls $B$ in $M$.
\end{defi}

It follows from the previous definition that a graph with polynomial growth is a
doubling measure space with respect to the counting measure. Since metric spaces
admitting a doubling measure are doubling metric spaces (\cite{gromov},
B.3) we may conclude that graphs with polynomial growth form a family of
doubling metric graphs.

\paragraph{Cayley Graphs of Nilpotent Groups}
Let $G$ be a finitely generated group with generating set $H$. Assume that the identity element $e \notin H$. We may associate the Cayley graph $\Ga(G,H)$ of G with respect to $H$,
whose vertices are the elements of $G$. The set of edges $E(G,H)$ is defined as follows,

\[
E(G,H)=\{ (g,gh) | g\in G, h \in H\cup H^{-1}\}.
\]

Then, define the growth of the group $G$ as the growth of the Cayley graph
$\Ga(G,H)$ with respect to some (any) generating set $H$. The order of growth
is well defined, since changing the generating set $H$ only changes the
constants appearing in the bounds of the growth function.

Let $G$ be a group. If $H_1, H_2$ are subgroups of $G$, define $[H_1,H_2]$ to be the subgroup of $G$ gennerated by all commutators $\{[h_1,h_2]:=h_1 h_2 h_1^{-1} h_2^{-2} | h_1 \in H_1, h_2 \in H_2 \}$. For $n \in \mathbb{N}$ we inductively define $C_n(G)$ by

\[
C_0(G):=G \ \mbox{and} \ \forall \ n \in \mathbb{N} \ C_{n+1}(G):=[G,C_n(G)].
\]
The group $G$ is nilpotent if, there exists $n \in \mathbb{N}$ such that $C_n(G)$ is the trivial group. A group $G$ is almost nilpotent if it contains a nilpotent normal subgroup of finite index.

% For further details see, for instance, \cite{Nilpotent}.

\begin{teo}[Gromov \cite{gromov} and Wolf \cite{wolf}] \label{nilpotent} A finitely generated
group has polynomial growth if and only if it is almost nilpotent.
\end{teo}

It follows from Theorem \ref{nilpotent} above that Cayley graphs of nilpotent groups have polynomial growth which in turns implies that they form a family of doubling metric graphs.

\begin{rem}\label{RemarkPansu}{\rm
A concrete example of a group with polynomial growth is  the discrete
Heisenberg group $\mathcal{H}_3(\bbZ)$, where

\[
 \mathcal{H}_3(\bbZ):=\left\{ \left( \begin{array}{ccc}
            1 & x& z \\
            0 & 1 & y\\
            0 & 0& 1 \\
                     \end{array}\right): x,y,z \in \bbZ
   \right\}.
\]
Since the discrete Heisenberg group is nilpotent, it has polynomial growth. Indeed, $\mathcal{H}_3(\mathbb{Z})$ has polynomial growth of order $4$. For further details we refer the reader to \cite{gromov}.

Observe that subgroup of $\mathcal{H}_3(\mathbb{Z})$ generated by
\[
 H:=\left\{\left( \begin{array}{ccc}
            1 & 1& 0 \\
            0 & 1 & 0\\
            0 & 0& 1 \\
                     \end{array}\right), \left( \begin{array}{ccc}
            1 & 0& 1 \\
            0 & 1 & 0\\
            0 & 0& 1 \\
                     \end{array}\right)
   \right\}
\]
is isomorphic to $\mathbb{Z}^2$. Then it follows from Theorem \ref{T1} and Theorem \ref{LyonsandYuval} that the discrete Boolean model of percolation on the Cayley graph of  $\mathcal{H}_3(\mathbb{Z})$ exhibits {\it{phase transition}} under the assumption that the random radius of a ball has finite fourth moment.

It is worth mentioning that, by a discrete version of Pansu Theorem \cite{pansu1989metriques}, {\it{the discrete Heisenberg group}} is an example of a doubling metric space which {\it{cannot be embedded in $\bbR^n$}}, for any $n$. Therefore, the absence of percolation in doubling graphs does not follow from the subcriticality of the Boolean percolation model in $\mathbb{R}^n$ and standard coupling arguments.
}
\end{rem}

\paragraph{Self-Similar Graphs}
A large family of a doubling metric graph are the self-similar graphs.   Self-similar graphs can be seen as discrete versions of  self-similar sets.

Let $\Gamma=\left(V(\Gamma),E(\Gamma)\right)$ be a graph with vertex set $V(\Gamma)$ and edge set $E(\Gamma)$. Let $F$ be a set of vertices in $\V \Ga$. Then $\cc_{\Ga}F$ denotes
the set of connected components in $\V \Ga\setminus F$. We define the
\emph{reduced graph} $\Ga_{F}$ of $\Ga$ by setting $\V \Ga_{F}=F$ and
connecting two vertices $x$ and $y$ in $\V \Ga_{F}$ by an edge if
and only if there exists a $C\in\cc_{\Ga}F$ such that $x$ and $y$
are in the boundary of $C$.

\begin{defi}\label{ss_in_growth}  We say that $\Ga$ is \emph{self-similar} with
respect to $F$ and $\psi:\V \Ga\to\V \Ga_{F}$ if
\begin{enumerate}
\item[(F1)] no vertices in $F$ are adjacent in $\Ga$,
\item[(F2)] the intersection of the closures of two different components in $\cc_{\Ga}F$
contains not more than one vertex and
\item[(F3)] $\psi$ is an isomorphism between $\Ga$ and $\Ga_{F}$.
\end{enumerate}
\end{defi}

We say that a graph has  bounded geometry if the set of vertex degrees is bounded.

\begin{teo}[Kr{\"o}n \cite{kron2004growth}]The Assouad dimension of homogeneously self-similar
graphs of bounded geometry are finite and equal to the Hausdorff dimension
of self-similar sets.
\end{teo}

\begin{figure}[h]
\centering
\includegraphics[width=6cm]{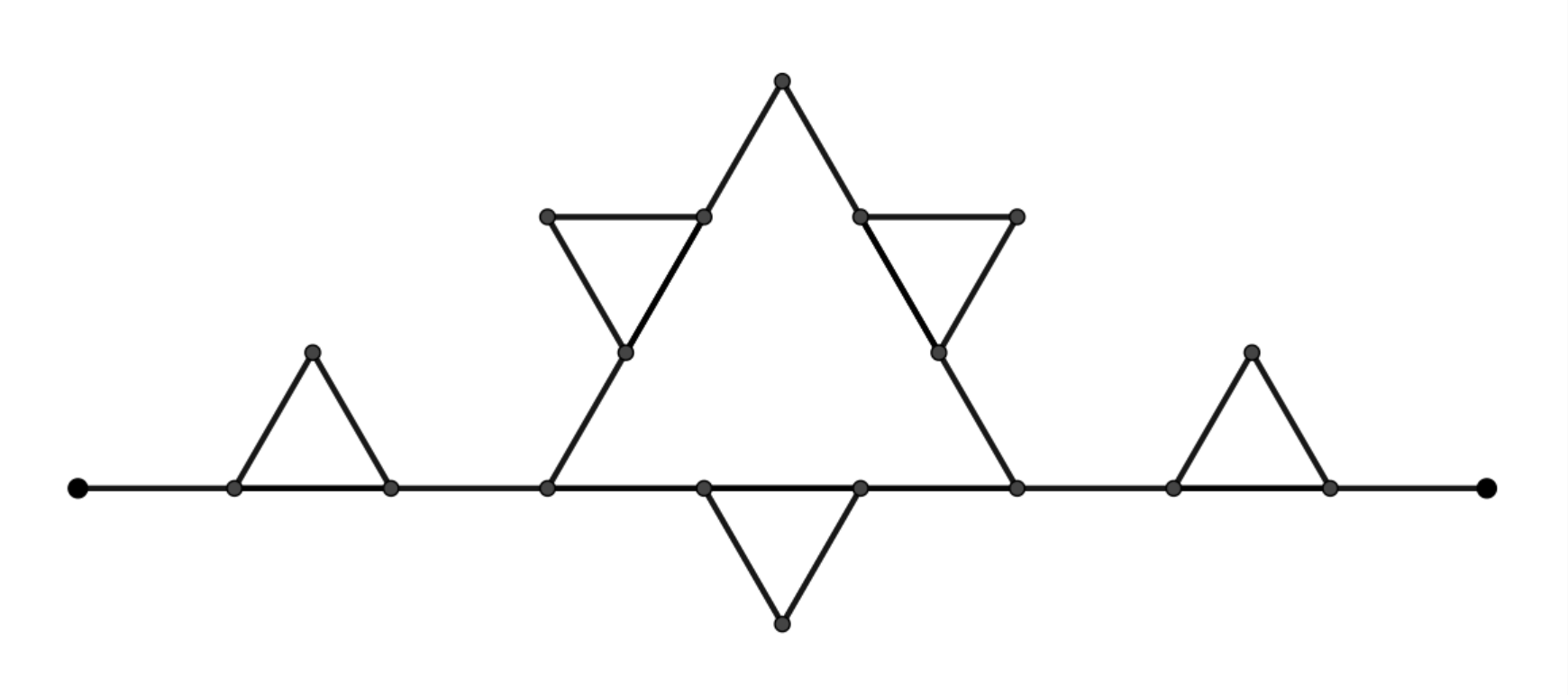}
\caption{The modified Koch curve is an homogeneously self-
similar  graph with  bounded geometry and Hausdorff and Assouad dimension  $\log 5/\log 3 $ \cite{kron2004growth}.}
\end{figure}

\section{Proof of Theorem \ref{T1}}
\label{mainp}
The proof of Theorem \ref{T1} will be divided into two steps. In the first step
we introduce two families of events, $F(v,r)$ and $H(v,r)$, in order to study the
diameter of the cluster $C(v)$. The family of events $F(v,r)$ is helpful to understand the behavior of the the diameter of the cluster $C(v)$ on the subgraph of $\cG(\cX,\cR)$
induced by the point process on $B(v,10r)$. The family of events $H(v,r)$ provides a way to take care of the influence of the point process $(\cX,\cR)$ from the exterior of the ball $B(v,r)$. Our aim in this step is to show that the probability of the
percolation event can be controlled by the probabilities of the events $F(v,r)$.
In the second step, we will show that if the radii are not too large, then the
occurrence of the event $F(v,r_1)$ implies the occurrence of two independent
events $F(u,r_2)$ and $F(w,r_2)$ where $r_1=10r_2$. Our aim in this step is to
show that the probability of the events $F(v,r_1)$ can be bounded from above by the square of the probability of the events $F(u,r_2)$ plus a quantity that goes to zero
when $r_1$ goes to infinity. This provides a way to take care of the probabilities
of the events $F(v,r)$ that allows us to show that for $p$ small enough,
$\P_{p,\,\rho}(F(v,r))$ goes to zero when $r$ goes to infinity.

\subsection{Controlling the diameter of the clusters $C(v)$}
For each $v\in V$, let $D_v=\inf\{r\geq 0:\, C(v)\subset B(v,r)\}$. The
percolation event is equivalent to the event $\bigcup_{v\in V}\{D_v=\infty\}$.
The proof of Theorem \ref{T1} is reduced to show that there exists $p_0>0$ such
that for each $v\in V$, $\lim_{r\to\infty}\P_{p,\,\rho}(D_v>r)=0$ for all
$p<p_0$.

For each $v\in V$ we define two families of events to study the diameter of the cluster $C(v)$.

\paragraph{The family of events $F(v,r)$} Let $B$ be a subset of $V$. Denote by
$\cG[B]$ the subgraph of $\cG(\cX,\cR)$ induced by $B$. Let $A$ be a non-empty subset of
$V$ contained in $B$ and let $v\in A$. We say that $v$ is
\emph{disconnected from the exterior of $A$ inside $B$} if the connected
component of $\cG[B]$ containing $v$ is contained in $A$. Now we introduce the events $F(v,r)$. Let $v\in V$ and let
$r\in\N$, we say that $F(v,r)$ does not occur if $B(v,r)$ is disconnected from the
exterior of $B(v,8r)$ inside $B(v,10r)$.

\paragraph{The family of events $H(v,r)$} For each $v\in V$ and $r\in\N$, we define
\begin{eqnarray}
\label{Hvr}
H(v,r)=\left\{\exists\,w\in\cP\cap B(v,10r)^c:\, R_w>\frac{d(w,v)}{10}\right\}.
\end{eqnarray}

The relation between the diameter of the cluster at $v$ and the families of events defined above is established in the following lemma.

\blem
\label{GHM}
The following inclusion holds for all $r\in\N$:
\begin{eqnarray}
\label{Mgrande2}
F(v,r)^c\cap H(v,r)^c\subset\left\{D_v\leq 8r\right\}.
\end{eqnarray}
\elem

\subsection*{Proof of Lemma \ref{GHM}}
If the event $H(v,r)$ does not occur, then there are no sites of the point process with distance to $v$ greater than $10r$ connected to the ball $B(v,9r)$. Indeed, assume that $H(v,r)$ does not occur. Then for every $w\in\cP\cap B(v,10r)^c$ we have $d(w,v)-R_w\geq\frac{9}{10}d(w,v)>9r$. Using the triangle inequality it is easy to verify that $d(u,v)\geq d(w,v)-R_w>9r$ for all $u\in B(w,R_w)$. If $F(v,r)$ does not occur, then the ball $B(v,r)$ is isolated from the exterior of $B(v,8r)$. If, in addition, the event $H(v,r)$ does not occur, then the balls $B(w,R_w)$, $w\in\cP\cap B(v,10r)^c$, do not  connect any vertex inside $B(v,r)$ to the complement of $B(v,8r)$. Thus $D_v\leq 8r$.
\hfill\square

\medskip

From \reff{Mgrande2} we get
\begin{eqnarray}
\label{Mgrande3}
\P_{p,\,\rho}(D_v>8r)\leq \P_{p,\,\rho}(F(v,r))+\P_{p,\,\rho}(H(v,r)).
\end{eqnarray}
In the following lemma we show that $\lim_{r\to\infty}\P_{p,\,\rho}(H(v,r))=0$ for all $p\in(0,1)$. Then we need to show that $\P_{p,\,\rho}(F(v,r))$ goes to zero as $r\to\infty$.

%%%%%%%%%%%%%%%%%%%%%%%%%%%%%%%%%%%%%%%%%%%%%%%%%%%%%%%%%%%%%%%%%%%%%%%%%%%%%%%%%%%%%%%%%%%%%%%%%%%%%%%%%%%%%%%%%%%%%%%%%%%
%%%%%%%%%%%%%%%%%%%%%%%%%%%%%%%%%%%%%%%%%%%%%%%%%%%%%%%%%%%%%%%%%%%%%%%%%%%%%%%%%%%%%%%%%%%%%%%%%%%%%%%%%%%%%%%%%%%%%%%%%%%

\blem
There exists a positive constant $C$ , which depends only on the Assouad dimension $\asdim$ of $\Ga$, such that for each $v\in V$ and $r\in\N$, the following inequality holds:
\begin{eqnarray}\label{0}
\P_{p,\,\rho}\left(H(v,r)\right) \leq C \ \E_{\rho}\left[R^{dim_A(\Gamma)}1\{R > r\}\right].
\end{eqnarray}
Furthermore, if $\E_{\rho}\left[R^{dim_A(\Gamma)}\right] < +\infty$, then $\displaystyle \lim_{r \rightarrow +\infty} \P_{p,\,\rho}\left(H(v,r)\right) = 0$.
\elem

\begin{proof}
%We claim that
%\[\P_{p,\,\rho}\left(H(v,r)\right) \leq \displaystyle C_1 \sum_{s > r}  s^{dim_A(\Gamma)}  \P_{\rho}(R = s)
%\]
%for some positive constant $C_1$. Since $\E_{\rho}\left[R^{dim_A(\Gamma)}\right] < +\infty$, we may conclude that $\displaystyle \lim_{r \rightarrow +\infty} \P_{p,\,\rho}\left(H(v,r)\right) = 0$.
%Indeed,
Let

\[
L_v(r) := \displaystyle \sum_{w \in B(v,10r)^c} 1\{w \in \mathcal{P}\} 1\{w \in B(v,10R_w)\}
\]
and note that $\P_{p,\,\rho}\left(H(v,r)\right) = \P_{p,\,\rho}\left(L_v(r) \geq 1\right)$. It follows that

\begin{eqnarray}
\P_{p,\,\rho}\left(H(v,r)\right) &\leq& \E_{p,\,\rho}\left[L_v(r)\right] \nonumber \\
&\leq& \E_{p,\,\rho}\left[\displaystyle \sum_{w \in B(v,10r)^c}  1\{w \in B(v,10R_w)\}\right] \nonumber \\
&=& \displaystyle \sum_{w \in B(v,10r)^c}  \E_{p,\,\rho}\left[1\{w \in B(v,10R_w)\}\right] \nonumber \\
&=& \displaystyle \sum_{w \in B(v,10r)^c} \sum_{s > r} 1\{w \in B(v,10s)\} \P_{\rho}(R_w = s) \nonumber \\
%&=& \displaystyle \sum_{\rho > 10 r} \sum_{w \in B(v,10\rho) \setminus B(v,10r)}  \mathbb{P}[R_w = \rho] \nonumber \\
&=& \displaystyle \sum_{s > r} \medida{B(v,10s) \setminus B(v,10r)}  \P_{\rho}(R = s) \nonumber \\
&\leq& \displaystyle \sum_{s > r} \medida{B(v,10s)}  \P_{\rho}(R = s) \nonumber \\
&\leq& C \displaystyle \sum_{s > r} s^{\asdim}  \P_{\rho}(R = s) \nonumber
%&=&\E_{p,\,\rho}[\displaystyle \sum_{w \in B(v,10r)^c}  1\{w \in B(v,10R)\}] \nonumber \\
%&=& \E_{\rho}[\medida{B(v,10R)\setminus B(v,10r)}] \nonumber \\
%&\leq& \displaystyle \sum_{\rho > 10 r} \medida{B(v,10\rho)}  \P_{\rho}\left(R = \rho\right) \nonumber \\
%&\leq& C \displaystyle \sum_{\rho > 10 r} \rho^{dim_A(\Gamma)}  \P_{\rho}\left(R = \rho\right). \nonumber
\end{eqnarray}
where $C=C_1 10^{\asdim}$. The last inequality and the choice of the constant $C$ follow from (\ref{constant}). This gives (\ref{0}).
\end{proof}
%%%%%%%%%%%%%%%%%%%%%%%%%%%%%%%%%%%%%%%%%%%%%%%%%%%%%%%%%%%%%%%%%%%%%%%%%%%%%%%%%%%%%%%%%%%%%%%%%%%%%%%%%%%%%%%%%%%%%%%%%%%%
%%%%%%%%%%%%%%%%%%%%%%%%%%%%%%%%%%%%%%%%%%%%%%%%%%%%%%%%%%%%%%%%%%%%%%%%%%%%%%%%%%%%%%%%%%%%%%%%%%%%%%%%%%%%%%%%%%%%%%%%%%%%

\subsection{Controlling the probabilities of the events $F(v,r)$}
To take care of the probabilities $\P_{p,\,\rho}(F(v,r))$ we introduce another family of events.

\paragraph{The family of events $\tilde{H}(v,r)$} For each $v\in V$
and $r\in\N$, we define
\begin{eqnarray}
\label{hrt}
\tilde{H}(v,r)=\{\exists\,w\in\cP\cap B(v,100r):\,R_w \geq r\}.
\end{eqnarray}

%We will use the following notation: $\Delta_{S{(x,10r)}}$ will denote the centers of a covering of $S{(x,10rd)}$ by balls of  of radius $rd$. We note that by doubling hypothesis  $\left\medida{\Delta_{S{(x,10r)}}\right}< C^4$

\blem The following inclusion holds for all $r\in\N$:
\label{Escala2}

\begin{eqnarray}
\label{escala} F(v,10r)\cap\tilde{H}(v,r)^c&\subset&
\bigcup_{u\in A(v,r,10)}F(u,r)
\bigcap
\bigcup_{w\in A(v,r,80)}F(w,r),
\end{eqnarray} \normalsize
where $A(v,r,m)$ is a $r$-separated subset of $S(v,mr)$ such that $\medida{A(v,r,m)}  =\newline  N(S(v,mr),r)$.
\elem

\subsection*{Proof of Lemma \ref{Escala2}}
\label{Geom3}
Fix $r\in\N$. First, assume that the event $F(v,10r)$ occurs but the event
$\tilde{H}(v,r)$ does not occur. Since $F(v,10r)$ occurs we can go from the
vertex $v$ to the complement of the ball $B(v,80r)$ just using balls $B(u,R_u)$
centered at points from $\cP\cap B(v,100r)$. In this way, we can go from the
sphere $S(v,10r)$ to the sphere $S(v,80r)$. One of these balls, let say
$B(u_*,R_{u_*})$, touches $S(v,10r)$. Since the sphere $S(v,10r)$ is a subset of
$\bigcup_{u\in A(v,r,10)}B(u,r)$ we get that this ball touches a ball of the
form $B(u,r)$ for some $u$ in $A(v,r,10)$.

Now we shall prove that, for this $u$, the event $F(u,r)$ occurs. It is easy to see that we can go from $B(u,r)$ to the complement of $B(u,8r)$ just using balls of the form $B(w,R_w)$ centered at points from $\cP\cap B(v,100r)$.Since $\tilde{H}(v,r)$ does not occur, the radius of any such ball is less than $r$. Then, we can go from $B(u,r)$  to the complement of $B(u,8r)$ just using balls of the form $B(w,R_w)$ centered at points from $\cP\cap B(u,10r)$. In other words, the event $F(u,r)$ occurs. Then, the event $\bigcup_{u\in A(v,r,10)}F(u,r)$ does occur. The proof that the event $\bigcup_{w\in A(v,r,80)}F(w,r)$ does occur follows in the same lines.
\hfill\square

\medskip

The event on the right side of \reff{escala} is the intersection of two events. The events do only depend on the restriction of the point process to the mentioned regions, and thus the events are independent.
%The first depends on what happens inside $B(v,20r)$. The other event only depends on what happens in the region $B(v,70r)^c$. Then, these two events are independent.
It follows from \reff{escala} that

%\begin{eqnarray}
%\P(F(v,10r))& \leq & \scriptstyle \left(\sum_{u\in A(v,r,10)}\P_{p,\,\nu}(F(u,r))\right)\left(\sum_{w\in A(v,r,80)}\P_{p,\,\nu}(F(w,r))\right)  \label{Ineq2}\\
%&+&\P_{p,\,\nu}(\tilde{H}(v,r)).\nonumber
%\end{eqnarray}
%\normalsize

\begin{eqnarray}
\P(F(v,10r))& \leq & \sum_{u\in A(v,r,10)}\P_{p,\,\rho}(F(u,r)) \nonumber \\
&\times & \sum_{w\in A(v,r,80)}\P_{p,\,\rho}(F(w,r))  \label{Ineq2}\\
&+&\P_{p,\,\rho}(\tilde{H}(v,r)).\nonumber
\end{eqnarray}

Notice that $\medida{A(v,r,m)}\leq C_1 m^{\asdim}$ for all $v\in V$, $r\in\N$ and $m\in\N$. Therefore, by \reff{Ineq2} we get
\begin{eqnarray} \textstyle
\sup_{v\in V}\P_{p,\,\rho}(F(v,10r))&\leq& K\left(\sup_{v\in V}\P_{p,\,\rho}(F(v,r))\right)^2\label{Ineq4}\\
&&+\sup_{v\in V}\P_{p,\,\rho}(\tilde{H}(v,r)),   \nonumber
\end{eqnarray}

where $K=C_{1}^{2}800^{\asdim}$.

\blem
\label{SB}
There exists positive constants $C_2$ and $C_3$, which depend only on the Assouad dimension $\asdim$ of $\Ga$, such that for each $v\in V$ and $r\in\N$, the following inequalities hold:
\begin{eqnarray}
\label{SB1}
\P_{p,\,\rho}(F(v,r))&\leq& p\, C_2r^{\asdim},\\
\label{SB2}
\P_{p,\,\rho}(\tilde{H}(v,r))&\leq& p\,C_3\E_{\rho}\left[R^{\asdim}\one\{R\geq r\}\right].
\end{eqnarray}
\elem

\subsection*{Proof of Lemma \ref{SB}}
Let $r\in\N$. A simple computation shows that
\begin{eqnarray*}
\P_{p,\,\rho}(F(v,r))&\leq&\P_{p,\,\rho}(\exists\,x\in\cP\cap B(v,10r))\nonumber\\
&\leq& p\,\medida{B(v,10r)}\\
&\leq& p\, C_2r^{\asdim},
\end{eqnarray*}
where $C_2=C_1 10^{\asdim}$. In the last inequality we used inequality (\ref{constant}).

To show \reff{SB2} we note that $\tilde{H}(v,r)=1\{Y_v\geq 1\}$, where $Y_v$ is a random variable defined by
\[Y_v=\sum_{u\in B(v,100r)}\one\{u\in\cP\}\one\{R_u\geq r\}.\]
We have
\begin{eqnarray*}
\P_{p,\,\rho}(\tilde{H}(v,r))&\leq&\E_{p,\,\rho}\left[Y_v\right]\nonumber\\
&=&\sum_{u\in B(v, 100r)}p\,\P_{\rho}(R_u\geq r)\nonumber\\
&=&p\,\medida{B(v,100r)}\P_{\rho}(R\geq r)\nonumber\\
&\leq& p\,C_3 r^{\asdim} \P_{\rho}(R\geq r)\nonumber\\
&\leq& p\,C_3\E_{\rho}\left[R^{\asdim}\one\{R\geq r\}\right],
\end{eqnarray*}
where $C_3=C_{1}100^{\asdim}$. The first equality follows from the independence between $\cP$ and $\cR$ and the second equality follows from the fact that the random variables $R_u, u\in V$ are identically distributed. The second inequality follows from inequality (\ref{constant}). \hfill\square

\subsection{Proof of Theorem \ref{T1}}
By \reff{Mgrande3}, the proof of Theorem \ref{T1} is reduced to show the existence of $p_0>0$ such that there exists an increasing sequence $(r_n)_{n\in\N}$ of natural numbers with $\displaystyle \lim_{n\to\infty} \P_{p,\,\rho}(F(v,r_n))$ $=0$ for all $p<p_0$, $v\in V$. For this reason we need the following lemma.

\blem
\label{FG0}
Let $f$ and $g$ be two functions from $\N$ to $\R_+$ satisfying the following conditions: (i) $f(r)\leq 1/2$ for all $r\in \{1,\dots, 10\}$; (ii) $g(r)\leq 1/4$ for all $r\in\N$; (iii) for all $r\in\N$:
\begin{eqnarray}
\label{FG1}
f(10r)\leq f^2(r)+g(r).
\end{eqnarray}
If $\lim_{r\to\infty}g(r)=0$, then $\lim_{n\to\infty}f(10^nr)=0$ for each $r\in\{1,\dots, 10\}$.
\elem

\subsection*{Proof of Lemma \ref{FG0}}
For each $n\in\N$, let  $F_n=\max_{1\leq r\leq 10}f(10^nr)$ and let $G_n=\max_{1\leq r \leq 10}g(10^nr)$. Using \reff{FG1} and hypothesis (i) and (ii) we may conclude, by means of the induction principle that, for each $n\in\N$,  $F_n\leq 1/2$ and
\begin{eqnarray}
\label{CotaIndb}
F_n \leq \frac{1}{2^{n+1}}+\displaystyle \sum_{j=0}^{n-1}\frac{1}{2^j}G_{n-1-j}.
\end{eqnarray}
Since $g(10^nr)$ goes to zero as $n\to\infty$ we have that $G_n\to 0$ when $n\to\infty$. By \reff{CotaIndb}, we obtain that $F_n\to 0$ when $n\to\infty$. \hfill\square

\medskip

Now we complete the proof of Theorem \ref{T1}. Consider the functions
$$f(r)=K\sup_{v\in V}\P_{p,\,\rho}(F(v,r))$$
and
$$g(r)=K\sup_{v\in V}\P_{p,\,\rho}(\tilde{H}(v,r)).
$$
It follows from \reff{Ineq4} that
\begin{eqnarray}
\label{Ineq5}
f(10r)\leq f^2(r)+g(r).
\end{eqnarray}

By condition \reff{E100f} and \reff{SB2} we have that $\lim_{r\to\infty}g(r)=0$ for any $p$.

We show that there exists $p_0>0$ such that if $p<p_0$ then $f(r)\leq 1/2$, $1\leq r\leq 10$ and $g(r)\leq 1/4$, $r\in\N$.

Set
\[p_0=\min\left(\frac{1}{2KC_210^{\asdim}}, \frac{1}{4KC_3\E_{\rho}\left[R^{\asdim}\right]}\right).\]

By condition \reff{E100f}, we get $p_0>0$.

Let $p>0$ be such that $p\leq p_0$. It follows from \reff{SB1} that
\begin{eqnarray*}
f(r)\leq\frac{1}{2}\left(\frac{r}{10}\right)^{\asdim}.
\end{eqnarray*}
Thus we have that if $0<p\leq p_0$, then $\max_{1\leq r\leq 10}f(r)\leq 1/2$.

By \reff{SB2}, we get
\begin{eqnarray*}
g(r)\leq\frac14.
\end{eqnarray*}

Finally, by Lemma \ref{FG0}, we have that $\lim_{n\to\infty}f(10^nr)=0$ for each $r\in\{1,\dots, 10\}$.
In particular
\[\lim_{n\to\infty}f(10^n)=\lim_{n\to\infty}K\sup_{v\in V}\P_{p,\,\rho}(F(v,10^n))=0.\]
\noindent This finishes the proof of Theorem \ref{T1}.
\hfill\square

%In next section by proving the complete coverage for almost-transitive graph whose growth is at most polynomial under the assumption $\E_{\rho}\left[R^{\asdim}\right] = \infty$.

\bigskip

%Before proving Theoren \ref{Riidinf}, we need to cite one further result which we state without proof. For further references see ... and references therein.

%\begin{teo}
%If $G=(V,E)$ is an infinite vertex-transitive, connected, locally finite graph with polynomial growth, then
%\[
%\frac{\medida{\partial S}}{\medida{S}} \geq \frac{2}{\mbox{diam}(S)+1}
%\]
%for every $S \subset V$.
%\end{teo}

\section{Proof of Theorem \ref{Riidinf}}\label{infr}
Fix some (any) vertex $v \in V$. We will prove that the following assertion holds:
\[\P_{p,\,\rho}(\exists\, w\in\cP: v \in B(w,R_w))=1.\]
Since
\begin{eqnarray*}
\P_{p,\,\rho}\left(\exists w \in \mathcal{P}: v \in B(w,R_{\omega})\right) &=& \P_{p,\,\rho}\left(\exists w \in \mathcal{P}: R_{w} > d(v,w)\right) \nonumber \\
&=& 1 - \P_{p,\,\rho}\left(\cap_{w \in V}(X_w=1,R_w > d(w,v))^c\right), \nonumber
\end{eqnarray*}
we will prove that
\begin{equation}\label{prodinf}
\P_{p,\,\rho}\left(\cap_{w \in V}(X_{\omega}=1,R_w > d(w,v))^c\right) = 0 .
\end{equation}
Since $\mathcal{X}=(X_w : w \in V)$ and $\mathcal{R}=(R_w:w \in V)$ are two families of independent random variables which are also independent between them, we have that (\ref{prodinf}) holds if and only if
\begin{equation}\label{inftyf}
\prod_{w \in V}\P_{p,\,\rho}\left((X_w=1,R_w > d(w,v))^c\right) = 0.
\end{equation}
It is well known that (\ref{inftyf}) holds if, and only if

\begin{eqnarray}\label{sum_inf}
\sum_{w \in V}\left(1-\P_{p,\,\rho}\left((X_w=1,R_w > d(w,v))^c\right)\right) &=& + \infty. \nonumber
\end{eqnarray}
%Since $\P_{p,\,\rho}[([X_{\omega}=1,R_{\omega} > d(\omega,v)]^c]=1-\mathbb{P}_{p,\nu}[([X_{\omega}=1,R_{\omega} > d(\omega,v)]]$, we show that $\sum_{\omega \in V}\P_{p,\,\rho}[([X_{\omega}=1,R_{\omega} > d(\omega,v)])] = + \infty$.
Since

\begin{eqnarray}
1-\P_{p,\,\rho}\left((X_w=1,R_w > d(w,v))^c\right) &=& \P_{p,\,\rho}\left(X_w=1,R_w > d(w,v)\right) \nonumber \\
&=& \P_{p,\,\rho}\left(X_w=1\right) \P_{p,\,\rho}\left(R_w > d(w,v)\right) \nonumber \\
&=& p \P_{\rho}\left(R_v > d(w,v)\right) \nonumber \\
&=& p \E_{\rho}\left[1\{\omega \in B(v,R_v)\}\right], \nonumber
\end{eqnarray}
we have that

\begin{eqnarray}
\sum_{w \in V}\left(1-\P_{p,\,\rho}\left((X_w=1,R_w > d(w,v))^c\right)\right) &=& p \sum_{w \in V} \E_{\rho}\left[1\{w \in B(v,R_v)\}\right] \nonumber \\
&=& p \E_{\rho}\left[\medida{B(v,R_v)}\right].
\end{eqnarray}
We conclude that $\P_{p,\,\rho}(\exists\, w\in\cP: v \in B(w,R_w))=1$ if, and only if $\E_{\rho}\left[\medida{B(v,R_v)}\right] = + \infty$. Since $\Gamma$ is a graph whose growth is at least polynomial, we get 
%from (\ref{almost-quasi-transitive})
that

\begin{eqnarray}
\E_{\rho}\left[\medida{B(v,R_v)}\right] &\geq& C_1 \E_{\rho}\left[R^C\right] \nonumber \\
&=& +\infty.
\end{eqnarray}
%Since $\E_{\rho}\left[R^{C}\right]= +\infty$, we conclude that $\P_{p,\,\rho}\left(\exists\, w\in\cP: v \in B(w,R_w)\right)=1$.

%%%%%%%%%%%%%%%%%%%%%%%%%%%%%%%%%%%%%%%%%%%%%%%%%%%%%%%%%%%%%%%%%%%%%%%%%%%%%%%%%%%%%%%%%%%%%%%%%%%%%%%%%%%%%%%%%%%%%%%%%%%%%%%%%%%%%%%%
%%%%%%%%%%%%%%%%%%%%%%%%%%%%%%%%%%%%%%%%%%%%%%%%%%%%%%%%%%%%%%%%%%%%%%%%%%%%%%%%%%%%%%%%%%%%%%%%%%%%%%%%%%%%%%%%%%%%%%%%%%%%%%%%%%%%%%%%%

\section{The Number of Infinite Clusters}
\label{ergodicitynumberofinfiniteclusters}

In this section we address the problem of determining {\it{how many infinite connected component there can be}}. We give an answer when the underlying graph has a family of symmetries which ``acts separating points''. We begin by recalling that an isometry on a graph
$\Ga$ with vertex set $V$ is a function $g: V \to V$ preserving the geodesic
distance of the graph, i.e., $d\left(g(v),g(w)\right)=d(v,w)$. We denote by
$\Iso(\Ga)$ the group of isometries of $\Ga$ and we observe that
isometries preserve the counting measure in $\Ga$.

We say that a family of isometries $S \subset \Iso(\Ga) $ \textit{acts  separating
points} of $\Ga$ if the orbit of any vertex of  $\Ga$ by the action of $S$ is infinite.

As a direct consequence of this  definition  we have also that $S$ separates compacts (i.e., finite sets). In other words, given a
compact set $K\in \Ga$, there exists $g\in S$ such that $g(K)\cap K = \emptyset$.

Let $(\cX,\cR)$ be a Bernoulli marked point process in $\Ga$ determined by a
family of random variables $\cX=(X_v:\, v\in V)$ and $\cR=(R_v: v\in V)$. We
say that an isometry $g:\Ga \to \Ga$ leaves the marked point  process
invariant if the random variables $R_{g(v)}$ and $R_{v}$ are
equally distributed. Henceforth we assume that the isometries $g\in S$ leave the marked point  process  invariant.

In order to state the result about ergodicity of the Boolean model we need to
define the action of the family of isometries on the marked point process.
For that purpose we assume that the
marked point process is defined in the space  of counting measures

\[
(\hat{\Ga}, \mathcal{A}):= (\mathcal{N}(\Ga\times \N_0), \bor{\Ga \times \N_0})
\]

\noindent where $\mathcal{N}(\Ga\times \N_0)$ is the set of locally finite counting  measures  in
$\Ga\times \N_0$. Let $\P_{p,\,\rho}$ be the distribution of a marked point process with retention parameter $p$.

For each isometry $g \in S$ leaving the marked point  process
invariant, we induce a map $\hat{g}:(\hat{\Ga}, \mathcal{A},\P_{p,\,\rho}) \rightarrow (\hat{\Ga}, \mathcal{A})$ as follows:

\[
\hat{g}(\omega)(B) = \omega(g^{-1}(B)),
\]

\noindent where  $B\in \bor{\Ga \times \N_0}$ e $\omega \in \hat{\Ga}$. The function  $\hat{g}$  is measurable and we observe that if $g$ leaves the process invariant then $\hat{g}$ is measure-preserving.

Let $\mathcal{I}$ denote the sigma-field of events that are invariant under all isometries of $S$, then the measure $\P_{p,\,\rho}$ is called $S$-\textit{ergodic} if for each $A\in \mathcal{A}$ , we have either $\P_{p,\,\rho}(A) = 0$ or $\P_{p,\,\rho}(A^c) = 0$.

We also note that the Boolean model is \textit{insertion tolerant}, i.e., 
$$\P_{p,\,\rho}(A\cup \{v\})>0$$ for every vertex $v$  and  every measurable $A$ determined by the marked point process $\left(\cX,\cR\right)$ with $\P_{p,\,\rho}(A)>0$. In the previous definition, we use that a vertex may be viewed as a ball of radius $0$ and  hence it can be identified with an element of $\bor{\Ga \times \N_0}$.
The insertion tolerant property follows from a direct comparison between the discrete Boolean percolation model and the underlying Bernoulli point process. In fact, a stronger property holds:

\[
\P_{p,\,\rho}(A\cup \{v\}) \geq p\P_{p,\,\rho}(A).
\]

Clearly, the definition of insertion tolerant may be generalized to any finite set of vertices $K$ and we may conclude that $\P_{p,\,\rho}(A\cup K)>0$ for any  finite subset  $K$ and any measurable set $A$ satisfying $\P_{p,\,\rho}(A)>0$. Then, an adaptation of the arguments given in \cite{yuval} yields:

\begin{teo}[Ergodicity of the Boolean model] \label{ergodicity}
Let $(\cX,\cR)$  be a Ber\-nou\-lli marked point process in a connected locally finite graph $\Ga$. Assume that  $\Ga$  has a family of isometries $S$ separating  points and leaving the marked point process $(\cX,\cR)$ invariant. Then the Boolean discrete percolation model is $S$-ergodic.
\end{teo}

\begin{proof}
Let $A$ be an $S$-invariant subset of $\Ga\times \N_0$, i.e., a set
satisfying
$\hat{g}A=A$, for all $g\in S$.
The idea is to show that $A$ is almost independent of $\hat{g}A$  for some $g$.

Let $\epsilon > 0$. Since $A$ is measurable, we may conclude from Theorem A.2.6.3 III in
\cite{verej}, that there exists a cylinder event
$B$ which depends only on some finite set $K$ such
that $\P_{p,\,\rho}(A \triangle B) < \epsilon$. For all $g\in S$, we have
$\P_{p,\,\rho}(\hat{g}A \Delta \hat{g}B) = \P_{p,\,\rho}[\hat{g}(A \Delta B)] <\epsilon$. By
assumption $S$ acts separating points, then there exists some $g\in S$ such
that $K$ and $gK$ are
disjoint. Since $gB$ depends only on $gK$, it follows that for some $g\in S$ the events 
$B$ and $gB$ are
independent. Thus,
  \begin{eqnarray*}
 |\P_{p,\,\rho} (A) - \P_{p,\,\rho} (A)^2 | &=& |\P_{p,\,\rho} (A\cap gA) - \P_{p,\,\rho} (A)^2 |\\
&\leq& |\P_{p,\,\rho} (A\cap gA) - \P_{p,\,\rho} (B\cap gA)| \\
&+& |\P_{p,\,\rho} (B\cap gA) - \P_{p,\,\rho} (B\cap
gB)| \\
&+& |\P_{p,\,\rho} (B\cap gB) - \P_{p,\,\rho} (B)2 | \\
&+& |\P_{p,\,\rho} (B)2 - \P_{p,\,\rho} (A)^2 |\\
&\leq& \P_{p,\,\rho} (A \triangle B) + \P_{p,\,\rho} (gA \triangle gB) \\
&+& |\P_{p,\,\rho} (B)\P_{p,\,\rho} (gB) -
\P_{p,\,\rho} (B)^2 |\\
&+& |\P_{p,\,\rho} (B) - \P_{p,\,\rho} (A)| \left(\P_{p,\,\rho} (B) + \P_{p,\,\rho} (A)\right) \\
&<& 4\epsilon .
\end{eqnarray*}
Therefore $\P_{p,\,\rho} (A) \in \{0, 1\}$ and the proof is complete.
\end{proof}

As a consequence of the ergodicity of the Boolean model we get the fo\-llo\-wing theorem.

\begin{teo}
Let $(\cX,\cR)$  be a Bernoulli marked point process in a connected locally finite graph $\Ga$. Assume that  $\Ga$  has a family of isometries $S$ separating  points and leaving the marked point process $(\cX,\cR)$ invariant. Then the number
of infinite clusters in the Boolean discrete percolation model is constant a.s. and equal either $0$, $1$, or $\infty$.
\end{teo}

\begin{proof}
 Let $N_{\infty}$ denotes the number of infinite clusters. The action of any element of $S$ on a configuration does not change the value $N_{\infty}$. In order to prove this assertion, consider $\hat{g} \in \hat{\Ga}$ and $\omega$ a realization of  the process. Then  $N_{\infty}(\hat{g}(\omega)(\Ga)) = N_{\infty}(\omega(g^{-1}(\Ga)) = N_{\infty}(\omega(\Ga))$. In other words, $N_{\infty}(\hat{g}(\omega)) = N_{\infty}(\omega)$. Hence, $N_{\infty}$ is measurable with respect to the sigma algebra of the $S$-invariants sets,  $\mathcal{I}_S$. Since the Boolean model is ergodic, we may conclude that $N_{\infty}$ is, a.s., constant.

Now, assume that  $N_C  = k \geq 2$. Let $v\in V$. Then, there exists an $R> 0$ such that $B(v,R)$  intersects all the $k$ infinite clusters.  It follows from the insertion tolerant property that the probability  of the number of infinite cluster being one is positive, contradicting the hypothesis that $N_C \geq 2$ a.s. The proof is complete.
\end{proof}

\addcontentsline{toc}{chapter}{Bibliography}
\bibliography{doubling}{}
\bibliographystyle{siam}

\end{document}